\documentclass[12pt,twoside]{article}
\usepackage[ansinew]{inputenc}
\usepackage{amsfonts}
\usepackage{latexsym}
\usepackage{amsmath}
\usepackage{amssymb}
\usepackage{eepic}
\usepackage{xcolor}
\usepackage{tikz}
\usepackage{float}
\usepackage{pstricks}
\usepackage{graphicx}
\usepackage{subfigure}

\newcommand{\R}{\mathbb{R}}

\newcommand{\C}{\mathbb{C}}

\newtheorem{defin}{Definition}%[section]

\newtheorem{theorem}{Theorem}

\newtheorem{exa}{Example}
\newenvironment{example}{\begin{exa}\rm}{\end{exa}}
\newtheorem{lemma}[defin]{Lemma}

\newenvironment{proof}
{\noindent{\it Proof.}}{\hfill $\Box$\par\vspace{2.5mm}}
\newenvironment{remark}
{\par\vspace{2.5mm}\noindent{\bf Remark.}}{\par\vspace{2.5mm}}

\newtheorem{que}{Question}

\newtheorem{pro}{Problem}

\numberwithin{equation}{section}

\makeatletter
\renewcommand{\ps@myheadings}{%
\renewcommand{\@evenhead}%
{{\rm\thepage}\hfil{\sc Gundersen, Heittokangas, Wen}\hfil}%
\renewcommand{\@oddhead}%
{\hfil{{\sc Contour integral solutions}\hfil{\rm\thepage}}}%
\renewcommand{\@evenfoot}{}%
\renewcommand{\@oddfoot}{}%
}\makeatother 

\title{\bf\Large Contour integral solutions of linear differential equations which include a generalization of the Airy integral}
\author{Gary~G.~Gundersen, Janne~M.~Heittokangas, Zhi-Tao~Wen*}
\date{}
\begin{document}
\maketitle

\begin{abstract}
The Airy integral is a well-known contour integral solution of Airy's equation which has several applications and which has been used for mathematical illustrations due to its interesting properties. We present and derive properties of two families of contour integral solutions of linear differential equations, where one family includes the Airy integral and Airy's equation, such that the family generalizes known properties of the Airy integral which include exponential decay growth in a certain sector. The second family includes a known example and contains a subfamily with interesting properties where a separate analysis of three pairwise linearly independent contour integral solutions of a particular equation is given. 

\bigskip



\end{abstract}

\renewcommand{\thefootnote}{}
\footnotetext[1]{*Corresponding Author: Zhi-Tao Wen (email: zhtwen@stu.edu.cn)}
\footnotetext[2]{The second author was supported by the Academy of Finland \#268009
and the Vilho, Yrj\"o and Kalle V\"ais\"al\"a Foundation of the Finnish Academy of Science and Letters. The third author was supported by the National Natural Science Foundation of China (No.~11971288 and No.~11771090) and Shantou University SRFT (NTF18029).}

\section{Introduction}  

Solutions of linear differential equations which are defined by contour integrals often have interesting properties. A well-known example is the Airy integral $\text{Ai}(z)$, which is a solution of Airy's equation $f^{\prime\prime} - zf = 0$, see Example~\ref{Airy ex} below. The Airy integral $\text{Ai}(z)$ has several applications (e.g., theory of diffraction, dispersion of water waves, turning point problem; see \cite{JJ}, \cite{SH}), and has also been used for mathematical purposes, such as to (a) illustrate results, (b) give examples on the ``boundary line'' of theorems, and (c) provide counterexamples; see \cite{Clarkson}, \cite{GLS}, \cite{HL}, \cite{HH}, \cite{Olver}, \cite{Steinmetz1}, \cite{Steinmetz2}. These references include the use of $\text{Ai}(z)$ for the purposes (a), (b), (c) in studies on the second Painlev\'{e} equation.

The Airy integral $\text{Ai}(z)$ is one of several interesting contour integral solutions of linear differential equations of the form 
\begin{equation}\label{lde}
f^{(n)}+P_{n-1}(z)f^{(n-1)}+\cdots +P_1(z)f'+P_0(z)f=0
\end{equation}
with polynomial coefficients $P_0(z),\ldots,P_{n-1}(z)$. In this paper we present and derive properties of two families of contour integral solutions of equations of the two respective forms
$$f^{(n)} +(-1)^{n+1} b f^{(k)} + (-1)^{n+1}zf = 0,   \quad\quad 0 < k < n, \quad b \in \C,$$ 
$$f^{(n)} - zf^{(k)} - f = 0,   \quad\quad  1 < k < n.$$ 

The first family includes the Airy integral $\text{Ai}(z)$ and  Airy's equation, and possesses properties that generalize known properties of $\text{Ai}(z)$, which include exponential decay growth in a certain sector, see Theorems~\ref{phi exist}-\ref{phi real} below. Since the Airy integral is a {\it special function} in mathematical physics that has several applications, this paper may be of interest in the areas of applied mathematics and special functions.

The second family also includes a known example. The properties of the second family include interesting features of a subclass of the family and a separate analysis of three pairwise linearly independent contour integral solutions of the equation $f^{(4)} - zf^{'''} - f = 0$ which are related by a specific identity, see Theorems~\ref{psi exist}-\ref{H and U} below.  

There are worthwhile questions to consider regarding contour integral solutions of equations of the form \eqref{lde} (including for the two families in this paper), which could lead to the creation of new special functions and new results in applied mathematics. We hope this paper and future studies on contour integral solutions will be used toward these goals and for the above purposes (a), (b), (c). 

The two families are discussed in Sections~\ref{phi family}-\ref{phi properties} and Sections~\ref{psi family}-\ref{specific example}, respectively. The next section contains some known examples.

\section{Known solutions from contour integrals} \label{known}

In general, every solution of \eqref{lde} is an entire function and each transcendental solution $f$ of \eqref{lde} has a finite rational order $d > 0$ and finite type $c > 0$, such that 
\begin{equation}\label{Valiron}
\log M(r, f) = (c + o(1))r^d
\end{equation}
as $r \to \infty$, where $M(r, f)$ is the maximum modulus function; see \cite[p. 108]{Valiron}. For possible values of $d$, see {\cite [Theorem 1] {GSW}}. Throughout the paper, $\rho(f)$ denotes the order of an entire function $f$.

\begin{example} \label{Airy ex} \cite{G}, \cite{IM}, \cite{JJ}, \cite{Olver}
The {\it Airy integral} $f(z) = \operatorname{Ai}(z)$ is a solution of the {\it Airy differential equation}
\begin{equation}\label{Airy DE}
 f'' - zf = 0
\end{equation}
that is defined by the contour integral
\begin{equation}\label{Airy}
\operatorname{Ai}(z) = \frac{1}{2\pi i}\int_C \exp\left\{\frac{1}{3}w^3 - wz \right\} dw,
\end{equation}
where the contour $C$ runs from $\infty$ to 0 along $\arg w = -\pi/3$ and then from 0 to $\infty$ along $\arg w = \pi/3$. The Airy integral $\text{Ai}(z)$ has the following properties: (i) $\rho(\text{Ai}) = 3/2$, (ii) $|\text{Ai}(z)|$ has exponential decay growth of order $\exp\{{-|z|^{-3/2}}\}$ in the sector $-\pi/3 < \arg z < \pi/3$, and (iii) $\text{Ai}(z)$ has an infinite number of negative real zeros and no other zeros, plus $\text{Ai}(z)$ has less than the usual overall frequency of zeros. We generalize property (i) in Theorem~\ref{phi order} and generalize property (ii) in Theorem~\ref{exp decay}.

If $\beta_1, \beta_2, \beta_3$ are the three distinct cube roots of unity, and if we set
\begin{equation} \label{Ai comp}
f_j(z) = \text{Ai}(\beta_j z), \quad j = 1, 2, 3,
\end{equation}
then the functions $f_1, f_2, f_3$ are three pairwise linearly independent solutions of \eqref{Airy DE}. Thus, every solution of \eqref{Airy DE} can be expressed as a contour integral function. We generalize property \eqref{Ai comp} in Theorem~\ref{fj independent}. 

In Sections~\ref{phi family}-\ref{phi properties} we present and discuss a family of equations and solutions which includes Airy's equation and the Airy integral.

\end{example}

\begin{example} \label{GS ex} $\cite{GS}$ A generalization of the Airy integral is as follows. For each positive integer $q$, there exist contour integral functions of the form
\begin{equation}\label{Airy gen}
g(z) = \frac{1}{2\pi i}\int_{D} e^{P_q(z, w)} dw,
\end{equation}
where $P_q(z, w)$ is a polynomial in $z$ and $w$, and where the contour $D$ can be chosen to consist of two rays that are connected in a similar way as in Example 1, such that $f = g(z)$ is a solution of the equation
\begin{equation}\label{gen Airy eq}
f^{\prime\prime} - z^qf = 0.
\end{equation}
When $q = 1$, \eqref{gen Airy eq} is Airy's equation \eqref{Airy DE} and the contour $D$ can be chosen so that $g(z)$ in \eqref{Airy gen} reduces to the Airy integral $\text{Ai}(z)$ in \eqref{Airy}. Generalizing Example~\ref{Airy ex}, for each $q$, the functions $g(z)$ in \eqref{Airy gen} have the properties that $|g(z)|$ has exponential decay growth in a certain sector and $g(z)$ has less than the usual frequency of zeros.  

\end{example}

\begin{remark} Example~\ref{GS ex} gives a generalization of the Airy integral in one way and this paper gives a generalization of the Airy integral in another way.\end{remark}

\begin{example} \label{GSW ex 1} \cite[Example 1] {GSW} For any $R > 0$, the contour integral function $G(z)$ defined by
\begin{equation}\label{G def}
G(z) = \int_{|w| = R} \exp\left\{\frac{z}{w} - \frac{1}{2w^2} - w \right\} dw
\end{equation}
is a solution of the equation
\begin{equation*}
f^{\prime\prime\prime} - zf^{\prime\prime} - f = 0,
\end{equation*}
where $\rho(G) = 1/2$. Since it is known {\cite [Corollary 2] {GSW}} that every transcendental solution $f$ of \eqref{lde} satisfies $\rho(f) \geq 1/(n - 1)$, the contour integral solution $G(z)$ is on the ''boundary line'' of this result, since $\rho(G) = 1/2 = 1/(n - 1)$ when $n = 3$. In Section~\ref{psi family} we present and discuss a family of equations and solutions that includes this example as a special case. 

\end{example}

\begin{example} \label{GSW ex 7} \cite[Example 7] {GSW}  Let the function $H(z)$ be defined by the contour integral
\begin{equation} \label{H def}
H(z) = \int_K \exp\left\{\frac{z}{\sqrt{2w}} - \frac{1}{4w} - w \right\} dw,
\end{equation}
where the contour $K$ is defined by $K = K_1 + K_2 + K_3$ with
\begin{eqnarray*}
K_1&:&  w = re^{i\pi/4}, ~ r ~\text{goes~from}~ +\infty ~ \text{to} ~ 1,\\
K_2&:&  w = e^{i\theta}, ~\theta~ \text{goes~from}~ \pi/4 ~ \text{to}~ 7\pi/4,\\
K_3&:&  w = re^{i7\pi/4}, ~ r~ \text{goes~from}~ 1 ~\text{to}~ +\infty,
\end{eqnarray*}
and where $\sqrt{2w}$ is defined by the branch
$$\sqrt{\zeta} = \exp\left\{\frac{1}{2}\log|\zeta| + i \frac{1}{2}\arg{\zeta}\right\}, \quad \quad 0 < \arg{\zeta} < 2\pi.$$
Then $f = H(z)$ is a solution of the equation
\begin{equation*} \label{H de}
f^{(4)} - zf^{\prime\prime\prime} - f = 0,
\end{equation*}
where $\rho(H) = 2/3$. In Section~\ref{specific example} we discuss another contour integral function $U(z)$ that solves this equation, and show that $U(z)$, $H(z)$, $H(-z)$ are three pairwise linearly independent solutions that satisfy a particular identity, see Theorem~\ref{H and U}. In this process, a more convenient form for $H(z)$ is derived. 

\end{example}

\section{A family that includes the Airy integral} \label{phi family}

Consider a linear differential equation of the form
	\begin{equation} \label{phi eq}
	f^{(n)} +(-1)^{n+1} bf^{(k)} +(-1)^{n+1} zf = 0, \quad\quad 0 < k < n, \quad b \in \C,
	\end{equation}
and let $\varphi = \varphi(n, k, b)$ denote the contour integral function
	\begin{equation} \label{phi def}
	\varphi(z) = \frac{1}{2 \pi i}\int_{C} \exp\left\{-wz + b\alpha w^{k+1} + \beta w^{n+1}\right\} \; dw,
	\end{equation}
where the contour $C$ runs from $\infty$ to 0 along $\arg w=-\pi/(n+1)$ and then from $0$ to $\infty$ along $\arg w=\pi/(n+1)$, and where
$$\alpha=\frac{(-1)^{k+1}}{k+1} \quad\quad \hbox{and} \quad\quad \beta=\frac{1}{n+1}.$$ 

In the particular case when $n = 2$ and $b = 0$, $\varphi(z)$ in \eqref{phi def} is the Airy integral $\text{Ai}(z)$ in \eqref{Airy} and \eqref{phi eq} is Airy's equation \eqref{Airy DE}. In the general case, we will show that $\varphi = \varphi(n, k, b)$ in \eqref{phi def} is a transcendental solution of \eqref{phi eq} with concrete properties that generalize known properties of the Airy integral. These properties include (i) exponential decay growth of $|\varphi(z)|$ in a certain sector (Theorem~\ref{exp decay}), and, (ii) when $b = 0$, all solutions of \eqref{phi eq} can be expressed in terms of contour integral functions that have a similar form to \eqref{phi def} (Theorem~\ref{fj independent}). 

The next section contains the aforementioned properties of $\varphi$. Most of this section is devoted to showing that $\varphi(z) \not\equiv 0$. 

\begin{theorem} \label{phi exist}
The function $f = \varphi(z)$ in \eqref{phi def} is a nontrivial solution of \eqref{phi eq}.
\end{theorem}

We use the following lemma in our proof of Theorem~\ref{phi exist}.

\begin{lemma} \label{inc}
Let $a\geq 0$ and $\eta>0$ be real constants. Suppose that $E(x)$ is a continuous real-valued function on $[0, \infty)$ for which there exists an $x_0>0$ such that $E'(x)>\eta$ for $x\geq x_0$. Then
\begin{equation} \label{-infty}
\int_0^{\infty} x^a E(x)\sin x \; dx = -\infty.
\end{equation}
\end{lemma}

\begin{proof} Set $v(x) = x^aE(x)$. For each integer $j \geq 0$, we have
\begin{equation*}
	\begin{split}
	\int_{2j\pi}^{(2j+2)\pi}v(x)\sin x\,dx&=\int_{2j\pi}^{(2j+1)\pi} v(x)\sin x\,dx+\int_{(2j+1)\pi}^{(2j+2)\pi} v(x)\sin x\,dx\\
	&=\int_{2j\pi}^{(2j+1)\pi} (v(x)-v(x+\pi))\sin x\,dx.
	\end{split}
\end{equation*}
Choose a positive integer $j_0$ such that $ 2\pi j_0 \geq x_0$, and set $D = -\eta\pi(2\pi j_0)^a$. Then for each $j \geq j_0$, when $x$ satisfies $2j\pi <x<(2j+1)\pi$, it follows from the mean value theorem that there exists a point $\xi_j \in (x, x+\pi)$ such that
    $$
    v(x)-v(x+\pi) \leq x^aE(x)-x^aE(x+\pi) = -\pi E'(\xi_j)x^a< D< 0.
    $$
Therefore,
\begin{equation*}
  \begin{split}
  \int_{2\pi j_0}^{\infty} x^a E(x) \sin x \; dx &= \sum_{j = j_0}^{\infty} \int_{2j\pi}^{(2j+1)\pi} (v(x) - v(x + \pi)) \sin x \; dx\\
  &< \sum_{j = j_0}^{\infty} D \int_{2j\pi}^{(2j+1)\pi} \sin x \; dx = \sum_{j=j_0}^{\infty} 2D = -\infty.
  \end{split}
\end{equation*}
It follows that \eqref{-infty} holds.\end{proof}

\bigskip

\noindent \emph{Proof of Theorem~\ref{phi exist}.}  Since $0 < k < n$, we obtain from \eqref{phi def},
\begin{equation*}
\begin{split}
&2 \pi i \{\varphi^{(n)}(z) +(-1)^{n+1} b\varphi^{(k)}(z) + (-1)^{n+1} z\varphi(z) \} \\
=&\int_{C} \left\{(-w)^n+(-1)^{n+1}b(-w)^k+(-1)^{n+1}z\right\}\exp\left(-wz+b \alpha w^{k+1}+\beta w^{n+1}\right) \, dw\\
=& \; (-1)^n \int_C \frac{\partial}{\partial w}\exp\left(-wz+b \alpha w^{k+1}+\beta w^{n+1}\right) \, dw\\
=& \; (-1)^n\left[\exp\left(-wz+b \alpha w^{k+1}+\beta w^{n+1}\right)\right]_{C}=0.
\end{split}
\end{equation*}
It follows that $f = \varphi(z)$ is a solution of \eqref{phi eq}. To complete the proof of Theorem~\ref{phi exist}, it suffices to show that $\varphi(z) \not\equiv 0$.  

Consider the value of $\varphi^{(p)}(0)$, where $p \geq 0$ is an integer. Set 
$$\gamma = \frac{k+1}{n+1} \pi  \quad\quad \hbox{and} \quad\quad b = |b|e^{i\lambda},$$
where $\lambda \in \R$. From \eqref{phi def},
	\begin{equation*}
	\begin{split}
	\varphi^{(p)}(0)=&\frac{(-1)^p}{2 \pi i}\int_{C} w^p \exp\left\{b\alpha w^{k+1}+\beta w^{n+1}\right\} \; dw\\
	=&\frac{(-1)^p}{2 \pi i}\int_{\infty}^0 r^p \exp\left\{|b|\alpha r^{k+1}e^{(\lambda-\gamma) i}-\beta r^{n+1} -\frac{p+1}{n+1} \pi i \right\} \; dr\\
	    &+\frac{(-1)^p}{2 \pi i}\int^{\infty}_0 r^p \exp\left\{|b|\alpha r^{k+1}e^{(\lambda+\gamma) i}-\beta r^{n+1} +\frac{p+1}{n+1} \pi i\right\} \; dr.
	\end{split}
	\end{equation*}
By taking real parts, we obtain
	\begin{equation} \label{p real part}
    \begin{split}
	\text{Re}\left(\varphi^{(p)}(0)\right) =& \frac{(-1)^p}{2\pi}\int_0^\infty r^p e_1(r) \sin\left\{|b|\alpha r^{k+1}\sin(\gamma-\lambda) +\frac{p+1}{n+1} \pi \right\}\,dr\\
         &+\frac{(-1)^p}{2\pi}\int_0^\infty r^p e_2(r) \sin\left\{|b|\alpha r^{k+1}\sin(\gamma+\lambda) +\frac{p+1}{n+1} \pi \right\}\,dr,
    \end{split}
	\end{equation}
where
    \begin{equation*}
    \begin{split}
    e_1(r) &= \exp\left\{|b|\alpha r^{k+1}\cos(\gamma-\lambda)-\beta r^{n+1}\right\},\\
    e_2(r) &= \exp\left\{|b|\alpha r^{k+1}\cos(\gamma+\lambda)-\beta r^{n+1}\right\}.
    \end{split}
    \end{equation*}
		
\bigskip

Suppose first that $b = 0$. Then from \eqref{p real part},
$$\text{Re}(\varphi(0)) = \frac{1}{\pi} \sin \left(\frac{\pi}{n+1}\right) \int_0^{\infty} \exp \left\{-\beta r^{n+1}\right\} \; dr > 0,$$
which proves that $\varphi(z) \not\equiv 0$.

Next suppose that $b \not= 0$. We now make the assumption that $\varphi(z) \equiv 0$. Then $\text{Re}\left(\varphi^{(p)}(0)\right) = 0$ for every $p \geq 0$. By choosing $p = n + j(n + 1)$ for $j = 0, 1, 2, \ldots$, it follows from \eqref{p real part} and $\text{Re}\left(\varphi^{(p)}(0)\right) = 0$ that
\begin{equation} \label{j}
\int_0^\infty r^n r^{j(n + 1)} u(r) \,dr = 0, \quad\quad j = 0, 1, 2, \ldots,
\end{equation}
where
$$u(r) = e_1(r)\sin\left(|b|\alpha r^{k+1}\sin(\gamma-\lambda)  \right) + e_2(r)\sin\left(|b|\alpha r^{k+1}\sin(\gamma+\lambda)  \right).$$
Let
\begin{equation} \label{ps}
\exp\left\{2\beta r^{n+1} \right\} = \sum_{j=0}^{\infty} c_j r^{j(n+1)}
\end{equation}
represent the power series expansion of $\exp\left\{2\beta r^{n+1}\right\}$ where the constants $c_j$ are the Taylor coefficients. From \eqref{j} and \eqref{ps}, we obtain
$$\int_0^\infty r^n \exp\left\{2\beta r^{n+1}\right\} u(r) \,dr = \sum_{j=0}^{\infty} c_j \int_0^\infty r^n r^{j(n+1)} u(r) \,dr = 0,$$
which yields
\begin{equation*}
\begin{split}
&\int_0^{\infty} r^n \exp\left\{|b|\alpha r^{k+1} \cos(\gamma-\lambda)+\beta r^{n+1} \right\} \sin\left(|b|\alpha r^{k+1} \sin(\gamma-\lambda) \right) \; dr\\
+&\int_0^{\infty} r^n \exp\left\{|b|\alpha r^{k+1} \cos(\gamma+\lambda)+\beta r^{n+1} \right\} \sin\left(|b|\alpha r^{k+1} \sin(\gamma+\lambda) \right) \; dr=0.
\end{split}
\end{equation*}
Noting that the sum of these two integrals equals zero, we have three cases.\\

{\it Case 1.} Suppose that $\sin(\gamma-\lambda)\sin(\gamma+\lambda)>0$. Then we have
\begin{equation*}
\begin{split}
&\int_0^{\infty} r^n \exp\left\{|b|\alpha r^{k+1} \cos(\gamma-\lambda)+ \beta r^{n+1} \right\} \sin\left(|b\alpha\sin(\gamma-\lambda)| r^{k+1}  \right) \; dr\\
+&\int_0^{\infty} r^n \exp\left\{|b|\alpha r^{k+1} \cos(\gamma+\lambda)+\beta r^{n+1} \right\} \sin\left(|b\alpha\sin(\gamma+\lambda)|r^{k+1} \right) \; dr=0.
\end{split}
\end{equation*}
By using the change of variables
\begin{equation}\label{ch of var}
x = |b\alpha\sin(\gamma-\lambda)| r^{k+1}  \quad \hbox{and} \quad  y = |b\alpha\sin(\gamma+\lambda)| r^{k+1},
\end{equation}
we deduce that
\begin{equation} \label{cv-1}
L\int_0^{\infty} x^a \exp \left\{cx + dx^{\mu} \right\} \sin x \; dx +M\int_0^{\infty} y^a \exp \left\{sy + ty^{\mu} \right\} \sin y \; dy = 0,
\end{equation}
where
\begin{equation*}
\begin{split}
\mu &= \frac{n+1}{k+1}, \quad\quad  L = \frac{1}{|\sin(\gamma-\lambda)|^{\mu}},  \quad\quad  a = \frac{n-k}{k+1},\\  
c& = \frac{\alpha\cos(\gamma-\lambda)}{|\alpha\sin(\gamma-\lambda)|},  \quad\quad  d = \frac{\beta}{|b\alpha\sin(\gamma-\lambda)|^{\mu}},  \quad\quad  M = \frac{1}{|\sin(\gamma+\lambda)|^{\mu}},\\
s& = \frac{\alpha\cos(\gamma+\lambda)}{|\alpha\sin(\gamma+\lambda)|},  \quad\quad  t = \frac{\beta}{|b\alpha\sin(\gamma+\lambda)|^{\mu}}.
\end{split}
\end{equation*}
Since $\mu > 1$, $d > 0$ and $t > 0$, the two integrands in \eqref{cv-1} both satisfy the conditions of the integrand in \eqref{-infty}. Hence, from Lemma~\ref{inc}, 
$$
L\int_0^{\infty} x^a \exp \left\{cx + d x^{\mu} \right\} \sin x \; dx +M\int_0^{\infty} y^a \exp \left\{sy + ty^{\mu} \right\} \sin y \; dy  = -\infty,
$$
which contradicts \eqref{cv-1}. Therefore, Case 1 cannot occur.\\

{\it Case 2.} Suppose that $\sin(\gamma-\lambda)\sin(\gamma+\lambda) < 0$. Then we have
\begin{equation*} 
\begin{split}
&\int_0^{\infty} r^n \exp\left\{|b|\alpha r^{k+1} \cos(\gamma-\lambda)+\beta r^{n+1} \right\} \sin\left(|b\alpha \sin(\gamma-\lambda)| r^{k+1}  \right) \; dr\\
=&\int_0^{\infty} r^n \exp\left\{|b|\alpha r^{k+1} \cos(\gamma+\lambda)+\beta r^{n+1} \right\} \sin\left(|b\alpha \sin(\gamma+\lambda)| r^{k+1} \right) \; dr.
\end{split}
\end{equation*}
As in Case 1, we use the change of variables \eqref{ch of var}, and deduce that
\begin{equation*} 
L\int_0^{\infty} x^a \exp \left\{cx + dx^{\mu} \right\} \sin x \; dx =M\int_0^{\infty} y^a \exp \left\{sy + ty^{\mu} \right\} \sin y \; dy,
\end{equation*}
which can be rewritten as
\begin{equation} \label{cv-2}
\int_0^{\infty} x^a F(x) \sin x \; dx = 0,
\end{equation}
where
$$F(x) = L\exp \left\{cx + dx^{\mu} \right\} - M\exp \left\{sx + tx^{\mu} \right\}.$$

Suppose that $c = s$ and $d = t$ both hold. Since $\sin(\gamma-\lambda)\sin(\gamma+\lambda) < 0$, we would then obtain
\begin{equation} \label{cos and sin rest}
\cos(\gamma-\lambda) = \cos(\gamma + \lambda) \quad \hbox{and} \quad \sin(\gamma - \lambda) = -\sin(\gamma + \lambda).
\end{equation}
It can be verified that \eqref{cos and sin rest} leads to a contradiction. Therefore, we cannot have both $c = s$ and $d = t$. 

Then from Lemma~\ref{inc}, it can be deduced that either
$$\int_0^{\infty} x^a F(x) \sin x \; dx = -\infty \quad\quad \hbox{or} \quad\quad \int_0^{\infty} x^a F(x) \sin x \; dx = +\infty$$
must hold, which contradicts \eqref{cv-2}. Thus, Case 2 cannot occur.\\

{\it Case 3.} Suppose that $\sin(\gamma-\lambda)\sin(\gamma+\lambda)=0$. Assume first that $\sin(\gamma-\lambda)=0$. Then it can be verified that $\sin(\gamma+\lambda) \not= 0$. We have 
\begin{equation*} 
\int_0^{\infty} r^n \exp\left\{|b|\alpha r^{k+1} \cos(\gamma+\lambda)+\beta r^{n+1} \right\} \sin\left(|b\alpha\sin(\gamma+\lambda)|r^{k+1} \right) \; dr=0.
\end{equation*}
By using the change of variable $x = |b\alpha\sin(\gamma+\lambda)| r^{k+1}$, we obtain
\begin{equation} \label{cv-3}
\int_0^{\infty} x^a \exp \left\{sx + tx^{\mu} \right\} \sin x \; dx = 0.
\end{equation}
But it follows from Lemma~\ref{inc} that
    $$
    \int_0^{\infty} x^a \exp \left\{sx + tx^{\mu} \right\} \sin x \; dx = -\infty,
    $$
which contradicts \eqref{cv-3}. By using the analogous argument, we also get a contradiction when $\sin(\gamma+\lambda)=0$. Hence, Case 3 cannot occur.\\

Since Cases 1, 2 and 3 cannot occur, we have a contradiction. Hence, our original assumption $\varphi(z) \equiv 0$ is false. Thus, $\varphi(z) \not\equiv 0$. The proof of Theorem~\ref{phi exist} is complete.\hfill$\Box$

\section{Properties of $\varphi(z)$} \label{phi properties}

From Theorem~\ref{phi exist}, the contour integral function $\varphi(z)$ is a nontrivial solution of equation \eqref{phi eq}. In this section we derive properties of $\varphi(z)$ which generalize properties of the Airy integral $\text{Ai}(z)$.   

\begin{theorem} \label{phi order}
The order of $\varphi$ in \eqref{phi def} satisfies $\rho(\varphi) = 1 + 1/n$.
\end{theorem}

\begin{proof} Observe first that \eqref{phi eq} cannot possess a nontrivial polynomial solution. Therefore, it follows from Theorem~\ref{phi exist} that $\varphi(z)$ must be transcendental. The assertion now follows from {\cite [Theorem 1] {GSW}}.
\end{proof}

Since the Airy integral $\text{Ai}(z)$ satisfies $\rho(\text{Ai}) = 3/2$ and has only negative real zeros, where the three critical rays of equation \eqref{Airy DE} are $\arg z = \pi, \pm\pi/3$, it follows from the asymptotic theory of integration that $|\text{Ai}(z)|$ has exponential decay growth of order $\exp\{-|z|^{3/2}\}$ on any ray in the sector $-\pi/3 < \arg z < \pi/3$, see \cite [Ch.~7.4] {Hille}, \cite{Olver}. This property of $\text{Ai}(z)$ is a particular case of the following general result. 

\begin{theorem} \label{exp decay}
Let $n \geq 2$. If $z = |z|e^{i\theta}$, where $\theta$ satisfies
\begin{equation} \label{decay sector}
-\frac{n\pi}{2n+2} < \theta < \frac{n\pi}{2n+2},
\end{equation}
then there exist positive constants $K = K(\theta)$ and $M = M(\theta)$ that depend only on $\theta$, such that $\varphi$ in \eqref{phi def} satisfies
\begin{equation} \label{phi exp decay}
|\varphi(z)| \leq M \exp\{-K|z|^{1+1/n}\}.
\end{equation}
\end{theorem}

\begin{remark} Two numbers in Theorem~\ref{exp decay} are best possible. Since $\rho(\varphi) = 1 + 1/n$ from Theorem~\ref{phi order}, it follows from \eqref{Valiron} and the Phragm\'en-Lindel\"of principle \cite{Cartwright} that (i) the exponent $1 + 1/n$ in \eqref{phi exp decay} could not be larger and (ii) the length $n\pi/(n + 1)$ of the interval \eqref{decay sector} could not be larger.\end{remark}

\noindent \emph{Proof of Theorem~\ref{exp decay}.}
We will use Cauchy's theorem twice to show that, for any $\theta$ satisfying \eqref{decay sector}, the contour $C$ in \eqref{phi def} can be replaced by two other contours (which depend on $\theta$) without changing the value of $\varphi(z)$ when $\arg z = \theta$. 

Let $\theta$ be a fixed constant satisfying \eqref{decay sector}. Then let $\mu = \mu(\theta)$ and $\tau = \tau(\theta)$ denote the constants
\begin{equation} \label{mu tau def}
\mu = -\frac{\pi}{n+1} - \frac{\theta}{n} \quad\quad \hbox{and} \quad\quad \tau = \frac{\pi}{n+1} - \frac{\theta}{n}.
\end{equation}
From \eqref{decay sector} and \eqref{mu tau def}, we obtain the following inequalities: 
\begin{equation} \label{mu rest}
-\frac{3\pi}{2} < (n+1)\mu < -\frac{\pi}{2} \quad\quad \hbox{and} \quad\quad -\frac{\pi}{2} < \mu + \theta < \frac{\pi}{2},
\end{equation}
\begin{equation} \label{tau rest}
\frac{\pi}{2} < (n+1)\tau < \frac{3\pi}{2} \quad\quad \hbox{and} \quad\quad -\frac{\pi}{2} < \tau + \theta < \frac{\pi}{2}.
\end{equation}

Next let $z$ be a point satisfying $\arg z = \theta$, i.e., set $z = |z|e^{i\theta}$. Then let $C(\mu)$ be the arc of the circle $|w| = |z|$ from $w = |z|e^{-i\pi/(n+1)}$ to $w = |z|e^{i\mu}$. From \eqref{decay sector}, \eqref{mu tau def}, \eqref{mu rest}, it can be deduced that
$$\cos(n + 1)\xi \leq \cos(n + 1)\mu < 0$$
holds for any point $w = |z|e^{i\xi}$ lying on $C(\mu)$. Therefore, with $\alpha$ and $\beta$ as in \eqref{phi def}, we obtain
\begin{equation*}
\begin{split}
&\left| \frac{1}{2\pi i} \int_{C(\mu)} \exp\left\{-wz+b\alpha w^{k+1}+\beta w^{n+1}\right\} \; dw \right|\\
&\leq \frac{|z|}{2\pi} \int_{C(\mu)} \exp\left\{|z|^2+|b\alpha||z|^{k+1}+\beta |z|^{n+1} \cos(n+1)\xi \right\} \; |d\xi|\\
&\leq |z| \exp\left\{|z|^2+|b\alpha||z|^{k+1}+\beta |z|^{n+1} \cos(n+1)\mu \right\}.
\end{split}
\end{equation*}
Since $2 \leq k + 1 < n + 1$, $\beta > 0$ and $\cos(n+1)\mu < 0$, we obtain
\begin{equation} \label{mu integral}
\frac{1}{2 \pi i} \int_{C(\mu)} \exp\left\{-wz+b\alpha w^{k+1}+\beta w^{n+1}\right\} \; dw  \to 0 \quad \hbox{as} \quad |z| \to \infty.
\end{equation} 
Similarly, by using \eqref{decay sector}, \eqref{mu tau def}, \eqref{tau rest}, we obtain
\begin{equation} \label{tau integral}
\frac{1}{2 \pi i} \int_{C(\tau)} \exp\left\{-wz+b\alpha w^{k+1}+\beta w^{n+1}\right\} \; dw  \to 0 \quad \hbox{as} \quad |z| \to \infty,
\end{equation}
where $C(\tau)$ denotes the arc of the circle $|w| = |z|$ from $w = |z|e^{i\pi/(n+1)}$ to $w = |z|e^{i\tau}$.

Now we observe from Cauchy's theorem that 
$$\frac{1}{2 \pi i} \int_{E_j} \exp\left\{-wz+b\alpha w^{k+1}+\beta w^{n+1}\right\} \; dw = 0,  \quad \quad j = 1, 2, $$
where $E_1$ is the closed curve consisting of the line segment from $0$ to $|z|e^{-i\pi/(n + 1)}$, the arc $C(\mu)$, and the line segment from $|z|e^{i\mu}$ to $0$, and $E_2$ is the closed curve consisting of the line segment from $0$ to $|z|e^{i\pi/(n + 1)}$, the arc $C(\tau)$, and the line segment from $|z|e^{i\tau}$ to $0$. Then by letting $|z| \to \infty$, it can be deduced from \eqref{phi def}, \eqref{mu integral} and \eqref{tau integral}, that 
\begin{equation} \label{L int}
\varphi(z)=\frac{1}{2 \pi i}\int_{L} \exp\left\{-wz+b\alpha w^{k+1}+\beta w^{n+1}\right\} \; dw,
\end{equation}
where the contour $L = L(\theta)$ runs from $\infty$ to $0$ along $\arg z = \mu$ and then from $0$ to $\infty$ along $\arg z = \tau$. Note that $L = C$ when $\theta = 0$.

\begin{figure}[H]\label{contour path L}
    \begin{center}
    \subfigure[when $\theta>0$]{
    \begin{tikzpicture}[scale=0.65]
    \draw[->](-2,0)--(5,0)node[left,below]{$x$};
    \draw[->](0,-3)--(0,4)node[right]{$y$};
    \draw[thick][domain=0:4] plot(\x,{-\x/2})node[right,font=\tiny]{$\arg w=-\pi/(n+1)$};
    \draw[thick][domain=0:3.5] plot(\x,{-\x})node[right,font=\tiny]{$\arg w=\mu$};
    \draw[thick][domain=0:4] plot(\x,{\x/2})node[right,font=\tiny]{$\arg w=\pi/(n+1)$};
    \draw[thick][domain=0:4] plot(\x,{\x/6})node[right,font=\tiny]{$\arg w=\tau$};
    \draw[thick][->](3,1.5)--(3.4,1.7);
    \draw[thick][->](4,-2)--(3.2,-1.6);
    \draw[thick][->](3.5,-3.5)--(3,-3);
    \draw[thick][->](3,3/6)--(3.1,3.1/6);
    \draw[-,dashed,thick](0.4,-0.3)--(2.6,-1.4);
    \draw[-,dashed,thick][->](0.4,-0.3)--(1.4,-0.8);
    \draw[-,dashed,thick](0.4,-0.3)--(2.1,-2);
    \draw[-,dashed,thick][->](1.614,-1.5)--(1,-0.9);
    \draw[-,dashed,thick](2.6,-1.4)arc(-30:-52:2)node[right=1pt,font=\tiny]{$E_1$};
    \draw[-,dashed,thick][->](2.6,-1.4)arc(-30:-45:2);
    \draw[-,dashed,thick](0.6,0.2)--(3,1.4);
    \draw[-,dashed,thick][->](0.6,0.2)--(1.6,0.7);
    \draw[-,dashed,thick](0.6,0.2)--(3.2,3.2/6+0.1);
    \draw[-,dashed,thick][<-](1.6,1.6/6+0.1)--(1.75,1.75/6+0.1);
    \draw[-,dashed,thick](3,1.4)arc(30:10:2.614)node[above=8pt,right=-4pt,font=\tiny]{$E_2$};
    \draw[-,dashed,thick][->](3,1.4)arc(30:18:2.614);
    \draw[thick](0,0)circle[radius=0.3pt]node[left=4pt,below]{$o$};
    \end{tikzpicture}
    }
    \subfigure[when $\theta<0$]{
    \begin{tikzpicture}[scale=0.65]
    \draw[->](-2,0)--(5,0)node[left,below]{$x$};
    \draw[->](0,-3)--(0,4)node[right]{$y$};
    \draw[thick][domain=0:4] plot(\x,{-\x/2})node[right,font=\tiny]{$\arg w=-\pi/(n+1)$};
    \draw[thick][domain=0:3.5] plot(\x,{\x})node[right,font=\tiny]{$\arg w=\tau$};
    \draw[thick][domain=0:4] plot(\x,{\x/2})node[right,font=\tiny]{$\arg w=\pi/(n+1)$};
    \draw[thick][domain=0:4] plot(\x,{-\x/6})node[right,font=\tiny]{$\arg w=\mu$};
    \draw[thick][->](1.2,0.6)--(3.4,1.7);
    \draw[thick][->](4,-2)--(2.4,-1.2);
    \draw[thick][<-](3,3)--(2,2);
    \draw[thick][<-](3,-3/6)--(3.1,-3.1/6);
    \draw[-,dashed,thick](0.4,0.3)--(3,1.6);
    \draw[-,dashed,thick][->](0.4,0.3)--(1.4,0.8);
    \draw[-,dashed,thick](0.4,0.3)--(2.4,2.3);
    \draw[-,dashed,thick][->](1.614,1.5)--(1,0.9);
    \draw[-,dashed,thick](3,1.6)arc(30:53:2.4)node[right=1pt,font=\tiny]{$E_2$};
    \draw[-,dashed,thick][->](3,1.6)arc(30:43:2.4);
    \draw[-,dashed,thick](0.6,-0.2)--(3,-1.4);
    \draw[-,dashed,thick][->](1.8,-1.8/2+0.1)--(2,-2/2+0.1);
    \draw[-,dashed,thick](0.6,-0.2)--(3.2,-3.2/6-0.1);
    \draw[-,dashed,thick][<-](1.6,-1.6/6-0.1)--(2,-2/6-0.1);
    \draw[-,dashed,thick](3,-1.4)arc(-30:-7:2.4)node[below=8pt,right=-4pt,font=\tiny]{$E_1$};
    \draw[-,dashed,thick][->](3,-1.4)arc(-28:-15:2.4);
    \draw[thick](0,0)circle[radius=0.3pt]node[left=4pt,below]{$o$};
    \end{tikzpicture}
    }
    \end{center}
    \begin{quote}
    \caption{Contours $L$, $C$ and closed curves $E_1$, $E_2$}
    \end{quote}
    \end{figure}
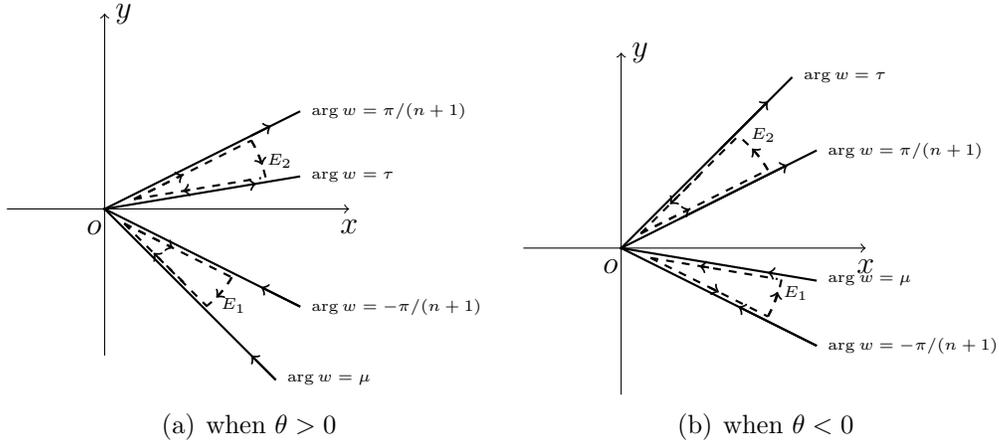

We next apply Cauchy's theorem to \eqref{L int}, and obtain further that for any constant $R > 0$,  
    \begin{equation}\label{phi new}
    \varphi(z)=\frac{1}{2 \pi i}\int_{J} \exp\left\{-wz+b\alpha w^{k+1}+\beta w^{n+1}\right\} \; dw,
    \end{equation}
where $J$ is the contour defined by $J=J_1+J_2+J_3$ with
    \begin{equation*}
    \begin{array}{cll}
    J_1&: w=re^{i\mu}, \quad  &r~\text{goes from}~\infty~\text{to}~R,\\
    J_2&: w=Re^{i\xi},  \quad &\xi~\text{goes from}~\mu~\text{to}~\tau, \\
    J_3&: w=re^{i\tau}, \quad  &r~\text{goes from}~R~\text{to}~\infty.
    \end{array}
    \end{equation*}
		
\begin{figure}[H]\label{contour path J}
    \begin{center}
     \subfigure[when $\theta>0$]{
    \begin{tikzpicture}[scale=0.65]
    \draw[->](-2,0)--(5,0)node[left,below]{$x$};
    \draw[->](0,-3)--(0,4)node[right]{$y$};
    \draw[-,dashed][domain=1.2:4] plot(\x,{-\x/2})node[right,font=\tiny]{$\arg w=-\pi/(n+1)$};
    \draw[thick][domain=0.94:3.5] plot(\x,{-\x})node[right,font=\tiny]{$\arg w=\mu$};
    \draw[-,dashed][domain=1.2:4] plot(\x,{\x/2})node[right,font=\tiny]{$\arg w=\pi/(n+1)$};
    \draw[thick][domain=1.3:4] plot(\x,{\x/6})node[right,font=\tiny]{$\arg w=\tau$};
    \draw[-,dashed][->](1.2,0.6)--(3,1.5);
    \draw[-,dashed][->](4,-2)--(2.4,-1.2);
    \draw[thick][->](3.5,-3.5)--(2,-2);
    \draw[-,dashed](0.9,-0.9)--(0,0);
    \draw[thick][->](0.95,-0.95)arc(-45:-25:2.3/2);
    \draw[-,dashed](0.85,-0.85)arc(-45:15:2.3/2-0.1);
    \draw[thick][->](2,2/6)--(3,3/6);
    \draw[-,dashed](1.2,1.2/6)--(0,0);
    \draw[thick](1.2,-0.6)arc(-30:15:2.3/2);
    \draw[-,dashed](1.3,1.3/6)arc(10:28:2.3/2);
    \draw[thick][->](1.2,-0.6)arc(-30:5:2.3/2)node[below=5pt,right,font=\tiny]{$|w|=R$};
    \draw[thick](0,0) circle [radius=0.3pt]node[left=4pt,below]{$o$};
    \end{tikzpicture}
    }
     \subfigure[when $\theta<0$]{
    \begin{tikzpicture}[scale=0.65]
    \draw[->](-2,0)--(5,0)node[left,below]{$x$};
    \draw[->](0,-3)--(0,4)node[right]{$y$};
    \draw[-,dashed][domain=1.2:4] plot(\x,{-\x/2})node[right,font=\tiny]{$\arg w=-\pi/(n+1)$};
    \draw[thick][domain=0.94:3.5] plot(\x,{\x})node[right,font=\tiny]{$\arg w=\tau$};
    \draw[-,dashed][domain=1.2:4] plot(\x,{\x/2})node[right,font=\tiny]{$\arg w=\pi/(n+1)$};
    \draw[thick][domain=1.3:4] plot(\x,{-\x/6})node[right,font=\tiny]{$\arg w=\mu$};
    \draw[-,dashed][->](1.2,0.6)--(3,1.5);
    \draw[-,dashed][->](4,-2)--(2.4,-1.2);
    \draw[thick][<-](3,3)--(2,2);
    \draw[-,dashed](0.9,0.9)--(0,0);
    \draw[thick][<-](0.95,0.95)arc(45:25:2.3/2);
    \draw[-,dashed](0.85,0.85)arc(45:-20:2.3/2-0.1);
    \draw[thick][<-](2,-2/6)--(3,-3/6);
    \draw[-,dashed](1.1,-1.1/6)--(0,0);
    \draw[thick](1.2,0.6)arc(30:-15:2.3/2);
    \draw[-,dashed](1.3,-1.3/6)arc(-10:-28:2.3/2);
    \draw[thick][<-](1.2,0.6)arc(30:-5:2.3/2)node[above=5pt,right,font=\tiny]{$|w|=R$};
    \draw[thick](0,0) circle [radius=0.3pt]node[left=4pt,below]{$o$};
    \end{tikzpicture}
    }
    \end{center}
    \begin{quote}
    \caption{Contour $J$}
    \end{quote}
    \end{figure}
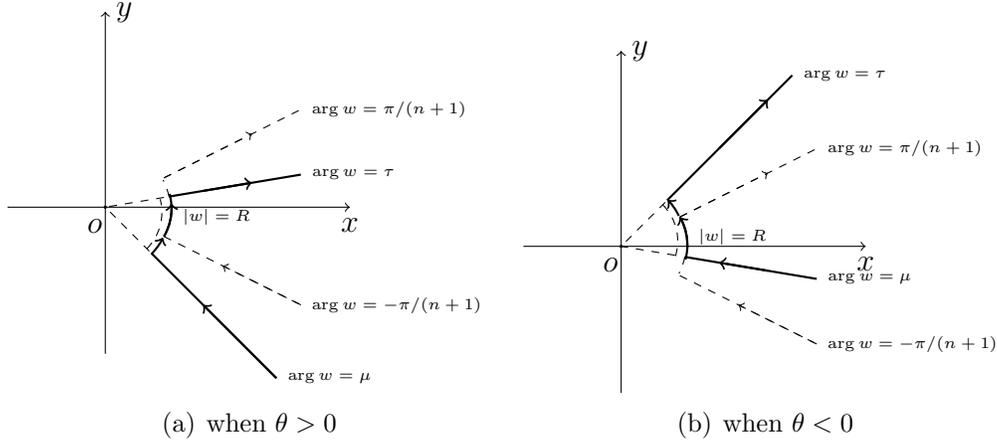

From \eqref{phi new}, 
	\begin{equation*}
	\begin{split}
	\varphi(z)=&\frac{1}{2 \pi i}\int_{\infty}^R \exp\left\{-r|z|e^{i(\mu+\theta)}
	+b\alpha r^{k+1}e^{i(k+1)\mu}+\beta r^{n+1}e^{i(n+1)\mu}\right\}\, e^{i\mu} dr\\
    &+\frac{1}{2 \pi}\int_{\mu}^{\tau}\exp\left\{-R|z|e^{i(\xi+\theta)}
    +b\alpha R^{k+1}e^{i(k+1)\xi}+\beta R^{n+1}e^{i(n+1)\xi}\right\} Re^{i\xi} \, d\xi\\
	  &+\frac{1}{2 \pi i}\int^{\infty}_R \exp\left\{-r|z|e^{i(\tau+\theta)}
	+b\alpha r^{k+1}e^{i(k+1)\tau}+\beta r^{n+1}e^{i(n+1)\tau} \right\}\, e^{i\tau} dr\\
    =&:I_1+I_2+I_3.
	\end{split}
	\end{equation*}	
Using \eqref{mu rest} and \eqref{tau rest}, let $A = A(\theta)$ denote the constant
\begin{equation} \label{A constant}
A = \min\{\cos(\mu+\theta), \; \cos(\tau+\theta)\} > 0.
\end{equation}
Then let $B = B(\theta)$ be any fixed constant that satisfies 
\begin{equation} \label{B}
0 < B < \{(n + 1)A\}^{1/n}.
\end{equation}
Choose $R=B|z|^{1/n}$. Since $\cos(\mu+\theta) > 0$ and $|\alpha| < 1 < 2\pi$, 
\begin{equation*}
    \begin{split}
    |I_1|&<\int_{B|z|^{1/n}}^\infty
    \exp\left\{-r|z|\cos(\mu+\theta)+|b|r^{k+1}+\beta r^{n+1}\cos(n+1)\mu \right\}\,dr\\
    &\leq \exp\left\{-B|z|^{1+1/n}\cos(\mu+\theta)\right\}
    \int_{B|z|^{1/n}}^\infty\exp\{|b|r^{k+1}+\beta r^{n+1}\cos(n+1)\mu\} \, dr.
    \end{split}
    \end{equation*}    
We have $\cos(n+1)\mu < 0$ from \eqref{mu rest}, and since $\beta > 0$ and $k < n$, 
\begin{equation} \label{I_1}
|I_1| < P\exp\left\{-B|z|^{1+1/n}\cos(\mu+\theta)\right\},
\end{equation}
where $P = P(\theta)$ is the following positive constant:
$$P = \int_0^\infty \exp\{|b|r^{k+1}+\beta r^{n+1}\cos(n+1)\mu\} \, dr.$$
Similarly, since $\cos(\tau + \theta) > 0$ and $\cos(n+1)\tau < 0$ from \eqref{tau rest}, we obtain
\begin{equation} \label{I_3}
|I_3| < Q \exp\left\{-B|z|^{1+1/n}\cos(\tau+\theta)\right\},
\end{equation}
where $Q = Q(\theta)$ is the following positive constant:
$$Q = \int_0^\infty \exp\{|b|r^{k+1}+\beta r^{n+1}\cos(n+1)\tau\} \, dr.$$
Last, for $|I_2|$, we have
$$|I_2| \leq \frac{B|z|^{1/n}}{2 \pi} \int_{\mu}^{\tau}
	\exp\left\{B(\beta B^n - \cos(\xi + \theta))|z|^{1+1/n} + |b|B^{k+1}|z|^{(k+1)/n}\right\}\,d\xi.$$
From \eqref{mu rest} and \eqref{tau rest}, 
$$-\frac{\pi}{2} < \mu + \theta \leq \xi + \theta \leq \tau + \theta < \frac{\pi}{2}, \quad \quad \mu \leq \xi \leq \tau.$$
Thus, from \eqref{A constant}, we see that
$$\cos(\xi + \theta) \geq A,  \quad \quad \mu \leq \xi \leq \tau.$$
Hence, we obtain
\begin{equation} \label{I_2}
|I_2| \leq B|z|^{1/n} \exp\left\{B(\beta B^n - A)|z|^{1+1/n}+|b| B^{k+1}|z|^{(k+1)/n}\right\},
\end{equation}
where $B(\beta B^n - A) < 0$ from $\beta = 1/(n+1)$ and \eqref{B}. Since $\varphi(z) = I_1 + I_2 + I_3$ when $\arg z = \theta$, we deduce that \eqref{phi exp decay} holds from an observation of \eqref{I_1}, \eqref{I_3}, \eqref{I_2}. This proves Theorem~\ref{exp decay}.\hfill$\Box$

\begin{remark} As stated above, Theorem~\ref{exp decay} generalizes the exponential decay growth property of the Airy integral $\text{Ai}(z)$ to $\varphi(z)$. Since all the zeros of $\text{Ai}(z)$ are real and negative, it is natural to ask whether the distribution of the zeros of $\varphi(z)$ also has some interesting feature(s)?\end{remark}

We now generalize another property of the Airy integral. Let $u(z)$ be the function $\varphi(z)$ in \eqref{phi def} when $b = 0$, that is,
\begin{equation} \label{u}
u(z) = \frac{1}{2\pi i} \int_{C}\exp\bigg\{-wz+\frac{1}{n+1}w^{n+1}\bigg\}\,dw,
\end{equation}
where the contour $C$ runs from $\infty$ to $0$ along $\arg w = -\pi/(n+1)$ and then from $0$ to $\infty$ along $\arg w = \pi/(n+1)$. Then from Theorems~\ref{phi exist} and \ref{phi order}, $u(z)$ is a transcendental solution of the equation 
\begin{equation} \label{b=0 eq}
f^{(n)} +(-1)^{n+1} zf = 0.
\end{equation}
When $n = 2$, $u(z)$ is the Airy integral $\text{Ai}(z)$ in \eqref{Airy} and \eqref{b=0 eq} is Airy's equation \eqref{Airy DE}.

If $\beta_1, \beta_2, \ldots , \beta_{n+1}$ are the $n+1$ distinct roots of unity, and if $f_j(z)$ denotes the function
\begin{equation} \label{fj def}
f_j(z) = u(\beta_j z), \quad j = 1, 2, \ldots , n+1,
\end{equation}
then $f_1, f_2, \ldots , f_{n+1}$ are solutions of equation \eqref{b=0 eq}. The following result is a generalization of the known property of the Airy integral $\text{Ai}(z)$ in \eqref{Ai comp}.

\begin{theorem} \label{fj independent}

The functions $f_1, f_2, \ldots , f_{n+1}$ in \eqref{fj def} are solutions of \eqref{b=0 eq} of order $1 + 1/n$, and any $n$ of these $n + 1$ functions are linearly independent.

\end{theorem}

Theorem~\ref{fj independent} shows that every solution of equation \eqref{b=0 eq} can be expressed as a contour integral function.\\ 

\noindent \emph{Proof of Theorem~\ref{fj independent}.} Since $u(z)$ in \eqref{u} is the function $\varphi(z)$ in \eqref{phi def} when $b = 0$, it follows from Theorem~\ref{phi exist}, Theorem~\ref{phi order} and \eqref{fj def}, that each $f_j$ is a solution of \eqref{b=0 eq} of order $1 + 1/n$. It remains to show that any $n$ of the $n + 1$ functions $f_1, f_2, \ldots , f_{n+1}$ are linearly independent. 

First we show that the $n$ functions $f_1, f_2, \ldots , f_n$ are linearly independent. Since $u$ is the function $\varphi$ in \eqref{phi def} when $b = 0$, it follows from Theorem~\ref{phi real} below that $u$ is real on the real axis. Hence from \eqref{p real part}, we obtain for $0 \leq p \leq n - 1$,
\begin{equation} \label{u(p)}
u^{(p)}(0)=\frac{(-1)^p}{\pi}\int_0^\infty r^p\exp\left(-\frac{1}{n+1} r^{n+1}\right)\sin\left(\frac{p+1}{n+1}\pi\right)dr \neq 0.
\end{equation}
Using \eqref{fj def} and the classical Vandermonde determinant, we calculate the Wronskian of $f_1, f_2, \ldots , f_n$ at $z = 0$:
    \begin{equation*}
    \begin{split}
    W(f_1, f_2,\ldots, f_n)(0)&=
    \begin{vmatrix}
    u(0)~&~u(0)~&\cdots&~u(0)\\
    \beta_1u'(0)~&~\beta_2u'(0)~&\cdots&~\beta_nu'(0)\\
    \vdots~&~\vdots~&~&\vdots\\
    \beta^{n-1}_1u^{(n-1)}(0)~&~\beta^{n-1}_2u^{(n-1)}(0)~&\cdots&~\beta^{n-1}_nu^{(n-1)}(0)\\
    \end{vmatrix}\\
    &=u(0)u'(0) \cdots u^{(n-1)}(0)
       \begin{vmatrix}
    1~&~1~&\cdots&~1\\
    \beta_1 ~&~\beta_2 ~&\cdots&~\beta_n \\
    \vdots~&~\vdots~&~&\vdots\\
    \beta^{n-1}_1 ~&~\beta^{n-1}_2 ~&\cdots&~\beta^{n-1}_n\\
    \end{vmatrix}\\
    &=u(0)u'(0)\cdots u^{(n-1)}(0)\prod_{1\leq i<j\leq n}(\beta_j-\beta_i) .
    \end{split}
    \end{equation*}
Since \eqref{u(p)} holds and since the $\beta_j$ are all distinct, we obtain
$$W(f_1, f_2,\ldots, f_n)(0) \not= 0.$$
Hence, the $n$ functions $f_1, f_2, \ldots , f_n$ are linearly independent. Similarly, the same argument will show that any $n$ of the $n + 1$ functions $f_j$ in \eqref{fj def} are linearly independent.\hfill$\Box$

\begin{remark} By varying the contour $C$ in \eqref{u}, we can express the functions $f_1, f_2, \ldots, f_{n+1}$ in \eqref{fj def} in a different form as follows. Let the functions $u_1, u_2, \ldots , u_{n+1}$ be defined by
\begin{equation} \label{uj}
u_j(z) = \frac{1}{2\pi i} \int_{C_j}\exp\bigg\{-wz+\frac{1}{n+1}w^{n+1}\bigg\}\,dw,
\end{equation}
where the contour $C_j$ runs from $\infty$ to 0 along $\arg w=-\frac{\pi}{n+1}+\frac{2\pi (j-1)}{n+1}$ and then
from 0 to $\infty$ along $\arg w=\frac{\pi}{n+1}+\frac{2\pi (j-1)}{n+1}$. We assume that $\beta_j = \exp\{i \frac{2\pi(j-1)}{n+1}\}$ for $j = 1, 2, \ldots , n + 1$. By using the change of variable $w=\beta_j\zeta$ in \eqref{uj}, we obtain from \eqref{u} and \eqref{fj def} that
$$u_j(z)=\frac{\beta_j}{2\pi i} \int_{C}\exp\bigg\{-\beta_j\zeta z+\frac{1}{n+1}\zeta^{n+1}\bigg\}\,d\zeta
=\beta_j u(\beta_jz)=\beta_jf_j(z)$$
for $j=1, 2, \ldots, n+1$. Thus for each $j$, the contour integral function $u_j(z)$ in \eqref{uj} is a non-zero constant multiple of the function $f_j(z)$ in \eqref{fj def}. 

Regarding solutions of \eqref{b=0 eq} that are defined by contour integrals, see page 414 of \cite{Heading}. By using a change of variable, it can be seen that, for each $n$, formula (53) in \cite{Heading} can be transformed into a non-zero constant mulitiple of \eqref{uj} above, where the constant multiple depends on $n$.  

\end{remark}

It is well known that the Airy integral $\text{Ai}(z)$ is real on the real axis. More generally, the next result shows that this property holds for $\varphi(z)$ whenever $b$ is real.

\begin{theorem} \label{phi real}
When $b \in \R$, the function $\varphi(z)$ in \eqref{phi def} is real on the real axis.
\end{theorem}

\begin{proof} Let $\varphi(z)$ be given by \eqref{phi def}, where $b \in \R$. For convenience, set
$$\gamma = \frac{k+1}{n+1}\pi \quad\quad \hbox{and} \quad\quad \eta = \frac{\pi}{n+1}.$$
Let $x\in\R$. From \eqref{phi def},
    \begin{equation*}
    \begin{split}
    \varphi(x) = & \frac{1}{2\pi i}\int_C \exp\left(-wx+b\alpha w^{k+1}+\beta w^{n+1}\right)\,dw\\
    = &\frac{1}{2\pi i}\int_\infty^0
    \exp\left(-rxe^{-\eta i}+b\alpha r^{k+1}e^{-\gamma i}-\beta r^{n+1}-\eta i\right)\, dr\\
       &+\frac{1}{2\pi i}\int_0^\infty
    \exp\left(-rxe^{\eta i}+b\alpha r^{k+1}e^{\gamma i}-\beta r^{n+1}+\eta i\right)\, dr\\
		= &-\frac{1}{2\pi i}\int_0^\infty \exp\left(A+iB\right)\, dr
     +\frac{1}{2\pi i}\int^\infty_0\exp\left(A-iB\right)\, dr,
    \end{split}
    \end{equation*}
where
    \begin{eqnarray*}
    A&=&-rx\cos\eta+b\alpha r^{k+1}\cos\gamma-\beta r^{n+1},\\
    B&=&rx\sin\eta-b\alpha r^{k+1}\sin\gamma-\eta.
    \end{eqnarray*}
Therefore, we obtain
$$\varphi(x) =\frac{1}{\pi}\int_0^\infty \exp(A)\sin(-B)\, dr,$$
which is real. This proves the assertion. \end{proof}

\section{Another family of contour integrals} \label{psi family}

Consider a linear differential equation of the form
\begin{equation} \label{psi eq}
f^{(n)} - zf^{(k)} - f = 0, \quad \quad 1 < k < n,
\end{equation}
and let $\psi = \psi(n, k)$ denote the contour integral function
\begin{equation} \label{psi def}
\psi(z) =  \frac{1}{2\pi i} \int_{|w| = 1} w^{k-2} \exp\left\{\frac{z}{w}- \frac{1}{(n-k+1)w^{n-k+1}} - \frac{w^{k-1}}{k-1} \right\} dw.
\end{equation}

\begin{theorem} \label{psi exist}
The function $f = \psi(z)$ is a transcendental solution of \eqref{psi eq}.
\end{theorem}

\begin{proof} From \eqref{psi def}, we obtain
\begin{equation*}
\begin{split}
&(2 \pi i)\{\psi^{(n)}(z) - z \psi^{(k)}(z) - \psi(z)\}\\
=&\int_{|w| = 1} \left(\frac{1}{w^{n-k+2}} - \frac{z}{w^2} - w^{k-2}\right) \exp\left\{\frac{z}{w}- \frac{1}{(n-k+1)w^{n-k+1}} - \frac{w^{k-1}}{k-1} \right\} dw\\
=&\int_{|w| = 1} \frac{\partial}{\partial w} \exp\left\{\frac{z}{w}- \frac{1}{(n-k+1)w^{n-k+1}} - \frac{w^{k-1}}{k-1} \right\} dw\\
=&\left[\exp\left\{\frac{z}{w}- \frac{1}{(n-k+1)w^{n-k+1}} - \frac{w^{k-1}}{k-1}  \right\}\right]_{|w| = 1} = 0,
\end{split}
\end{equation*}
since $|w| = 1$ is a closed curve. Thus, $\psi(z)$ satisfies \eqref{psi eq}.

Next we show that $\psi \not\equiv 0$. From \eqref{psi def} and $w = e^{i\theta}$, we obtain
\begin{equation} \label{k-1 & 0}
\psi^{(k-1)}(0) = \frac{1}{2\pi}\int_{-\pi}^{\pi} \exp\left\{-e^{i(k-1)\theta} \left(\frac{e^{-in\theta}}{n-k+1} + \frac{1}{k-1}\right) \right\} d\theta.
\end{equation}
Let $L = L(n, k, \theta)$ denote
\begin{equation} \label{L def}
L = -e^{i(k-1)\theta} \{\sigma e^{-in\theta} + \eta\},
\end{equation}
where 
$$\sigma = \frac{1}{n-k+1}  \quad \hbox{and} \quad \eta = \frac{1}{k-1}.$$ 
From \eqref{L def}, we have
\begin{equation} \label{gen L}
L = P(n, k, \theta) + i Q(n, k, \theta),
\end{equation}
where $P(n, k, \theta)$ and $Q(n, k, \theta)$ are the real-valued functions
$$P(n, k, \theta) = -\sigma \cos((n - k + 1)\theta) - \eta \cos((k-1)\theta),$$
$$Q(n, k, \theta) = \sigma \sin((n - k + 1)\theta) - \eta \sin((k-1)\theta).$$
From \eqref{k-1 & 0}, \eqref{L def}, \eqref{gen L},
\begin{equation} \label{PQ}
\psi^{(k-1)}(0) = \frac{1}{2\pi} \int_{-\pi}^{\pi} e^{P(n, k, \theta)} \{\cos Q(n, k, \theta) + i \sin Q(n, k, \theta) \} d\theta.
\end{equation}
Taking real parts of \eqref{PQ} gives
\begin{equation} \label{Re}
\text{Re}\left(\psi^{(k-1)}(0)\right) = \frac{1}{2\pi} \int_{-\pi}^{\pi} e^{P(n, k, \theta)} \cos Q(n, k, \theta) d\theta.
\end{equation}
Since $1 < k < n$, we obtain for all $\theta$,
\begin{eqnarray*}
|Q(n, k, \theta)| &\leq& \frac{|\sin ((n-k+1)\theta)|}{n-k+1} + \frac{|\sin(k - 1)\theta|}{k-1}\\
&\leq& \frac{1}{n-k+1} + \frac{1}{k-1} \leq \frac{1}{2} + 1 < \frac{\pi}{2}.
\end{eqnarray*}
Hence, $\cos Q(n, k, \theta) \geq \cos(3/2) > 0$ for all $\theta$. Then from \eqref{Re}, it follows that $\text{Re}\left(\psi^{(k-1)}(0)\right) > 0$. Thus, $\psi \not\equiv 0$. Furthermore, $\psi$ must be transcendental because it can be observed that \eqref{psi eq} cannot possess a nontrivial polynomial solution. This proves Theorem~\ref{psi exist}. \end{proof}

\begin{remark} Observe that Example~\ref{GSW ex 1} is a particular case of Theorem~\ref{psi exist} because when $n = 3$ and $k = 2$ in \eqref{psi def}, we obtain $2 \pi i \psi(z) \equiv G(z)$, where $G(z)$ is the function in \eqref{G def}. Example~\ref{GSW ex 7} is related to the particular case of Theorem~\ref{psi exist} when $n = 4$ and $k = 3$, see Section~\ref{specific example} below.\end{remark}   

The following result shows that the order of $\psi$ lies in $(0, 1)$.

\begin{theorem} \label{psi order}
The order of $\psi$ in \eqref{psi def} satisfies $\rho(\psi) = 1 - 1/k$.
\end{theorem}

\begin{proof} From \eqref{psi def} and Cauchy's theorem, it follows that
\begin{equation*}
\psi(z) =  \frac{1}{2\pi i} \int_{|w| = R} w^{k-2} \exp\left\{\frac{z}{w}- \frac{1}{(n-k+1)w^{n-k+1}} - \frac{w^{k-1}}{k-1} \right\} dw
\end{equation*}
holds for any $R > 0$. Since $1 < k < n$, we have
\begin{equation} \label{psi abs}
|\psi(z)| \leq \frac{1}{2\pi} \int_{|w| = R} |w|^{k-2} \exp\left\{\frac{|z|}{|w|} + \frac{1}{|w|^{n-k+1}} + |w|^{k-1} \right\} |dw|.
\end{equation}
Let $z$ be a point satisfying $|z| \geq 1$ and set $R = |z|^{1/k}$. Then from \eqref{psi abs},
\begin{equation*}
|\psi(z)| \leq |z|^{1-1/k} \exp\left\{2|z|^{1-1/k} + 1 \right\},
\end{equation*}
which gives $\rho(\psi) \leq 1 - 1/k$. Since $\psi$ is transcendental from Theorem~\ref{psi exist}, it follows from \cite[Theorem 1]{GSW} that the only possible orders of $\psi$ are $1 + 1/(n - k)$ and $1 - 1/k$. Since $n > k$ and $\rho(\psi) \leq 1 - 1/k$, we obtain $\rho(\psi) = 1 - 1/k$.\end{proof}

The proofs of the next two results use the following observation. If $s$ is a non-negative integer, then from \eqref{psi def},
$$\psi^{(s)}(0) = \frac{1}{2\pi i} \int_{|w| = 1} w^{k-2-s} \exp\left\{- \frac{1}{(n-k+1)w^{n-k+1}} - \frac{w^{k-1}}{k-1}  \right\} dw.$$
Thus, from the residue theorem, we obtain
\begin{equation} \label{Res eq}	
\psi^{(s)}(0) = \text{Res} \left\{w^{k-2-s}\sum_{j=0}^\infty \frac{(-1)^j w^{j(k-1)}}{j!(k-1)^j} \left(\frac{k-1}{(n-k+1)w^n}+ 1\right)^j, \; 0 \right\}.
\end{equation}

\begin{theorem} \label{psi even unbounded}
When $n$ is even and $k$ is odd, then $\psi$ is an even function that is unbounded on every ray from the origin.
\end{theorem}

\begin{proof}
Since $n$ even and $k$ is odd, for any odd positive integer $s$ and any non-negative integer $j$, the two integers $k - 2 - s$ and $j(n - k + 1)$ are both even integers. It follows that all the powers of $w$ in \eqref{Res eq} are even powers, which means that the residue must be zero. Therefore, from \eqref{Res eq}, $\psi^{(s)}(0) = 0$ for all odd integers $s$. Hence, from the Maclaurin series, $\psi(z)$ is an even function.

Now suppose that $\psi(z)$ is bounded on one particular ray $\arg z = \theta_0$. Since $\psi(z)$ is an even function, it follows that $\psi(z)$ is also bounded on the ray $\arg z = \theta_0 + \pi$. Since $\psi$ is bounded on these two rays where $(\theta_0 + \pi) - \theta_0 = \pi$, and since $\rho(\psi) \in (0, 1)$ from Theorem~\ref{psi order}, it follows from the Phragm\'en-Lindel\"of theorem that $\psi$ must be bounded in the whole complex plane. Hence, from Liouville's theorem, $\psi$ is a constant, which is a contradiction. Thus our assumption is false, and $\psi$ must be unbounded on every ray from the origin. \end{proof}

\begin{theorem} \label{psi real}
The function $\psi$ is real on the real axis. If $n$ is even and $k$ is odd, then $\psi$ is also real on the imaginary axis.
\end{theorem}

\begin{proof}
From \eqref{Res eq}, $\psi^{(s)}(0)$ is a real number for every non-negative integer $s$. Thus, from the Maclaurin series, $\psi$ is real on the real axis.

When $n$ is even and $k$ is odd, then $\psi(z)$ is an even function from Theorem~\ref{psi even unbounded}. Then from the reflection principle, it follows that $\psi$ is real on the imaginary axis.\end{proof}

\section{Solutions of $f^{(4)} - zf''' - f = 0$}  \label{specific example}

In this section we discuss contour integral functions that are solutions of the particular differential equation
\begin{equation} \label{U equa}
f^{(4)} - zf^{\prime\prime\prime} - f = 0.
\end{equation}
One such solution is when $n = 4$ and $k = 3$ in \eqref{psi def}, and for convenience, we denote this function by $U(z)$. Then from Theorem~\ref{psi exist},
\begin{equation} \label{U def}
U(z) = \frac{1}{2 \pi i} \int_{|w| = 1} w \exp\left\{\frac{z}{w}- \frac{1}{2w^2} - \frac{w^2}{2} \right\} dw
\end{equation}
satisfies equation \eqref{U equa}. Another contour integral function that satisfies \eqref{U equa} is $H(z)$ in \eqref{H def} from Example~\ref{GSW ex 7}.

First, properties of $U(z)$ will be discussed. Next we derive an alternate form for $H(z)$ that is more convenient than \eqref{H def}. This alternate form is then used to show that $U(z)$, $H(z)$, $H(-z)$ are three pairwise linearly independent solutions of \eqref{U equa} that satisfy a specific identity, see Theorem~\ref{H and U}. 

From Theorems~\ref{psi order}, \ref{psi even unbounded} and \ref{psi real}, $U(z)$ has order 2/3, is an even function, is unbounded on every ray from the origin, and is real on both the real and imaginary axis. The next result gives further properties of $U(z)$ on the imaginary axis. 

\begin{theorem} \label{U on rays} The following statements hold for the function $U(z)$ in \eqref{U def}.\\

\textnormal{(a)} $U(z)$ has its largest growth in both directions of the imaginary axis, and we have 
\begin{equation}\label{M(r, U)}
|U(ir)| = |U(-ir)| = M(r, U) = (c_0 + o(1))r^{2/3}
\end{equation}
as $r \to \infty$, where $c_0 > 0$ is some constant.\\ 

\textnormal{(b)} $U(z)$ does not have any zeros on the imaginary axis.

\end{theorem}

We use the following lemma in our proof of Theorem~\ref{U on rays}.

\begin{lemma} \label{Mac coeff}
The Maclaurin series of the function $U(z)$ in \eqref{U def} is given by
\begin{equation} \label{ms}
U(z) = \sum_{\nu=0}^{\infty} a_{2\nu}z^{2\nu},
\end{equation}
where
$$a_0 = -\frac{1}{2} \sum_{m=0}^{\infty} \frac{1}{4^m m!(m+1)!}$$
and
$$a_{2\nu} = \frac{(-1)^{\nu+1}}{2^{\nu-1} (2\nu)!} \sum_{m=0}^{\infty} \frac{1}{4^m m!(m+\nu-1)!}, \quad \quad  \nu \geq 1.$$
\end{lemma}

\begin{proof}
Applying the residue theorem to \eqref{U def} gives
$$U(z) = \text{Res} \left\{w \sum_{j=0}^\infty \frac{1}{j!2^j} \left(\frac{2z}{w} - \frac{1}{w^2} - w^2\right)^j, \; 0 \right\}.$$
For convenience, we rewrite this as follows:
\begin{equation} \label{U res}
U(z) = \text{Res} \left\{\sum_{j=0}^\infty \frac{(1 + (w^4 - 2zw))^j}{(-2)^j j!w^{2j-1}} , \; 0 \right\}.
\end{equation}
Contributions to this residue can only come from terms in $\left(1 + (w^4 - 2zw)\right)^j$ that contain a power $w^{2j-2}$. From the binomial expansions, we obtain
$$(1 + (w^4 - 2zw))^j = \sum_{m=0}^{j} A_m(w^4 - 2zw)^m, \quad \hbox{where} \quad A_m = \frac{j!}{m!(j-m)!},$$
and
$$(w^4 - 2zw)^m = \sum_{p=0}^{m} B_p z^p w^{4m-3p}, \quad \hbox{where} \quad B_p = (-2)^p\frac{m!}{p!(m-p)!}.$$

From the above, we see that the residue contributions in \eqref{U res} can only occur when $p, m, j$ are integers satisfying
\begin{equation} \label{p,m,j}
4m - 3p = 2j - 2  \quad \hbox{and} \quad 0 \leq p \leq m \leq j.
\end{equation}
The contribution $L(p, m, j)$ to the residue from three such integers $p, m, j$ is
\begin{equation} \label{L}
L(p, m, j) = \frac{A_mB_pz^p}{(-2)^jj!} = \frac{z^p}{(-2)^{j-p}(j-m)!p!(m-p)!}.
\end{equation}
From \eqref{p,m,j}, $p$ must be an even integer. It follows that only even powers of $z$ can appear in the Maclaurin series of $U(z)$, which is also known from the fact that $U(z)$ is an even function. 

Suppose that $p = 0$. Then from \eqref{p,m,j},
$$j = 2m + 1, \quad m \geq 0,$$
and from \eqref{L},
$$L(0, m, 2m + 1) = \frac{1}{(-2)^{2m+1}(m+1)!m!}.$$
It follows that
\begin{equation} \label{nu = 0}
a_0 = -\frac{1}{2} \sum_{m=0}^{\infty} \frac{1}{4^m (m+1)!m!}.
\end{equation}

Now suppose that $p = 2\nu > 0$, where $\nu \geq 1$. Then from \eqref{p,m,j},
$$j = 2m + 1 - 3\nu, \quad m \geq 3\nu - 1,$$
and from \eqref{L},
$$L(2\nu, m, 2m + 1 - 3\nu) = \frac{z^{2\nu}}{(-2)^{2m+1-5\nu}(m+1-3\nu)!(2\nu)!(m-2\nu)!}.$$
It follows that
$$a_{2\nu} = (-1)^{\nu+1} \frac{2^{5\nu-1}}{(2\nu)!} \sum_{m=3\nu-1}^{\infty} \frac{1}{4^m (m+1-3\nu)!(m-2\nu)!}, \quad \quad \nu \geq 1.$$
By setting $q = m + 1 - 3\nu$, we obtain
\begin{equation} \label{nu > 0}
a_{2\nu} = \frac{(-1)^{\nu+1}}{2^{\nu-1} (2\nu)!} \sum_{q=0}^{\infty} \frac{1}{4^q q!(q+\nu-1)!}, \quad \quad \nu \geq 1.
\end{equation}
Lemma~\ref{Mac coeff} follows from \eqref{nu = 0} and \eqref{nu > 0}.\end{proof}

\noindent \emph{Proof of Theorem~\ref{U on rays}.} If $\lambda$ is a real number, then from \eqref{ms}, we obtain
$$U(i\lambda) = - \sum_{\nu=0}^{\infty} |a_{2\nu}| \lambda^{2\nu},$$
which proves both part (b) and the first two equalities in \eqref{M(r, U)}. The last equality in \eqref{M(r, U)} follows from $\rho(U) = 2/3$ and \eqref{Valiron}.\hfill$\Box$

\begin{remark}
Although it is known that the order of $U(z)$ is $2/3$, it is not yet known what the type of $U(z)$ is. Below, $\tau(U)$ denotes the type of $U(z)$ with respect to $M(r, U)$. We show how $\rho(U)$ and $\tau(U)$ can be computed by using Lemma~\ref{Mac coeff} with well-known formulas for the Maclaurin coefficients \cite[Ch.~2]{Boas}. 

Using Stirling's formula,   
$$\nu! = (1 + o(1)) \sqrt{2\pi\nu}\left(\frac{\nu}{e}\right)^\nu$$
as $\nu \to \infty$, we obtain for all $\nu$ large enough,
	\begin{eqnarray*}
	|a_{2\nu}|&=&\frac{1}{2^{\nu-1}(2\nu)!} \sum_{m=0}^\infty \frac{1}{4^mm!(m+\nu-1)!}\\ 
	&\leq& \frac{2}{2^{\nu}(2\nu)!(\nu-1)!}\sum_{m=0}^\infty \left(\frac{1}{4}\right)^m\\ 
	&\leq& \frac{3\nu}{2^{\nu}(2\nu)!\nu!} \leq \left(\frac{e}{2\nu}\right)^{3\nu}.	
	\end{eqnarray*}
To get a lower bound for $|a_{2\nu}|$, we again use Stirling's formula, and obtain for all $\nu$ large enough,	
	\begin{equation*}
	|a_{2\nu}| \geq \frac{1}{2^{\nu-1}(2\nu)!(\nu-1)!} = \frac{2\nu}{2^{\nu}(2\nu)!\nu!} \geq 
	\frac{1}{5}\left(\frac{e}{2\nu}\right)^{3\nu}.
	\end{equation*}
Therefore, it can be seen that 
	$$
	\rho(U)=\limsup_{\nu\to\infty}\frac{(2\nu)\log(2\nu)}{-\log |a_{2\nu}|}
	=\frac23,	
	$$
which agrees with the known result. Then, using the notation $\rho = \rho(U) = 2/3$ and the fact that $\lim_{\nu\to\infty}\nu^\frac{1}{\nu}=1$, we obtain
	$$
	\tau(U)=(e\rho)^{-1}\limsup_{\nu\to\infty}(2\nu)|a_{2\nu}|^\frac{\rho}{2\nu}=\frac{3}{2e}\limsup_{\nu\to\infty}(2\nu)|a_{2\nu}|^\frac{1}{3\nu}=\frac32.
	$$	
\end{remark}

\bigskip

As discussed earlier, the contour integral function $H(z)$ in \eqref{H def} also satisfies equation \eqref{U equa} and Theorem~\ref{H and U} below exhibits relationships between the three solutions $U(z)$, $H(z)$, $H(-z)$ of \eqref{U equa}. In order to prove Theorem~\ref{H and U}, we first derive a more convenient form of $H(z)$ that does not involve a branch of the logarithm as in Example~\ref{GSW ex 7}. 

\begin{lemma} \label{H alternate form} 

The function $H(z)$ in \eqref{H def} can be written in the form
\begin{equation} \label{new H form} 
H(z) = \int_{A} w \exp\left\{\frac{z}{w}- \frac{1}{2w^2} - \frac{w^2}{2} \right\} dw,
\end{equation}
where the contour $A$ is defined by $A=A_1+A_2+A_3$ with
\begin{eqnarray*}
A_1&:&  w = r, ~ r ~\text{goes~from}~ +\infty ~ \text{to} ~ 1,\\
A_2&:&  w = e^{i\theta}, ~\theta~ \text{goes~from}~ 0 ~ \text{to}~ \pi,\\
A_3&:&  w = -r, ~ r~ \text{goes~from}~ 1 ~\text{to}~ +\infty.
\end{eqnarray*}

\end{lemma}

\begin{proof} By making the change of variable $v = \sqrt{2w}$ in \eqref{H def}, we obtain
\begin{equation*} 
H(z) = \int_{J} v \exp\left\{\frac{z}{v}- \frac{1}{2v^2} - \frac{v^2}{2} \right\} dv,
\end{equation*}
where $J$ is the contour defined by $J = J_1 + J_2 + J_3$ with
\begin{eqnarray*}
J_1:  &v& = re^{\pi i/8}, ~ r ~\text{goes~from}~ +\infty ~ \text{to} ~ \sqrt{2},\\
J_2:  &v& = \sqrt{2}e^{i\theta}, ~\theta~ \text{goes~from}~ \pi/8 ~ \text{to}~ 7\pi/8,\\
J_3:  &v& = re^{7 \pi i/8}, ~ r~ \text{goes~from}~ \sqrt{2} ~\text{to}~ +\infty.
\end{eqnarray*}

From Cauchy's theorem, it can be seen that the curve $J$ can be replaced by the curve $L$ below as follows:
\begin{equation} \label{L curve}
H(z) = \int_{L} v \exp\left\{\frac{z}{v}- \frac{1}{2v^2} - \frac{v^2}{2} \right\} dv,
\end{equation}
where $L$ is the contour defined by $L = L_1 + L_2 + L_3$ with
\begin{eqnarray*}
L_1:  &v& = re^{\pi i/8}, ~ r ~\text{goes~from}~ +\infty ~ \text{to} ~ 1,\\
L_2:  &v& = e^{i\theta}, ~\theta~ \text{goes~from}~ \pi/8 ~ \text{to}~ 7\pi/8,\\
L_3:  &v& = re^{7 \pi i/8}, ~ r~ \text{goes~from}~ 1 ~\text{to}~ +\infty.
\end{eqnarray*} 

    \begin{figure}[H]\label{contour L}
    \begin{center}
    \begin{tikzpicture}[scale=0.55]
    \draw[->](-4,0)--(4,0)node[left,below]{$x$};
    \draw[->](0,-3)--(0,4)node[right]{$y$};
    \draw[thick][domain=0.924:4] plot(\x,{0.414*\x})node[right,font=\tiny]{$L_1: \arg v=\pi/8$};
    \draw[thick][domain=-0.85:-4] plot(\x,{-0.414*\x})node[left,font=\tiny]{$L_3: \arg v=7\pi/8$};
    \draw[thick][->](2.5,2.5*0.414)--(2.2,2.2*0.414);
    \draw[thick][->](-2,2*0.414)--(-2.4,0.414*2.4);
     \draw[thick][->](0.924,0.924*0.414)arc(30:90:1);
     \draw[thick](0.924,0.924*0.414)arc(30:153:1)node[below=1pt,right=0.5pt,font=\tiny]{$L_2: |v|=1$};
     \draw[-,dashed][->](1.5,1.5*0.414)arc(30:90:1.6);
     \draw[-,dashed](1.5,1.5*0.414)arc(30:153:1.6)node[above=20pt,right=3.5pt,font=\tiny]{$J_2: |v|=\sqrt{2}$};
    \end{tikzpicture}
    \end{center}
    \begin{quote}
    \caption{Contour $L$}
    \end{quote}
    \end{figure}
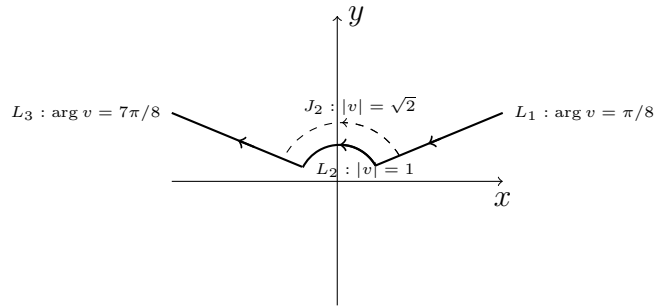 

We now show that the curve $L$ in \eqref{L curve} can be replaced by the curve $A$ in \eqref{new H form}, which will prove Lemma~\ref{H alternate form}. First we address the right half-plane. For any fixed $R > 1$, it follows from Cauchy's theorem that
\begin{equation} \label{D(R)}
\int_{D(R)} v \exp\left\{\frac{z}{v}- \frac{1}{2v^2} - \frac{v^2}{2} \right\} dv = 0,
\end{equation}
where $D(R)$ is the closed contour defined by $D(R) = D_1 + D_2 + D_3 + D_4$ with
\begin{eqnarray*}
D_1:  &v& = re^{\pi i/8}, ~ r ~\text{goes~from}~ R ~ \text{to} ~ 1,\\
D_2:  &v& = e^{i\theta}, ~\theta~ \text{goes~from}~ \pi/8 ~ \text{to}~ 0,\\
D_3:  &v& = r, ~ r ~ \text{goes~from}~ 1 ~\text{to}~ R,\\
D_4:  &v& = Re^{i\theta}, ~ \theta ~ \text{goes~from}~ 0 ~ \text{to}~ \pi/8.
\end{eqnarray*}
We now show that the contribution to the integral in \eqref{D(R)} from the arc $D_4$ goes to zero as $R \to \infty$. We have
$$\left| \int_{D_4} v \exp\left\{\frac{z}{v}- \frac{1}{2v^2} - \frac{v^2}{2} \right\} dv \right| \leq R^2 \int_{0}^{\pi/8} \exp\left\{\frac{|z|}{R}+\frac{1}{2R^2}-\frac{R^2}{2}\cos 2\theta \right\} d\theta.$$
Since $\cos 2\theta \geq \cos \pi/4 = \sqrt{2}/2$ for $0 \leq \theta \leq \pi/8$, we obtain
$$\left| \int_{D_4} v \exp\left\{\frac{z}{v}- \frac{1}{2v^2} - \frac{v^2}{2} \right\} dv \right| \leq \frac{\pi R^2}{8} \exp\left\{\frac{|z|}{R}+\frac{1}{2R^2}-\frac{\sqrt{2}R^2}{4} \right\} .$$
Therefore,
\begin{equation}\label{D4} 
\int_{D_4} v \exp\left\{\frac{z}{v}- \frac{1}{2v^2} - \frac{v^2}{2} \right\} dv \to 0 \quad \hbox{as} \quad R \to \infty.
\end{equation}
By letting $R \to \infty$ in \eqref{D(R)} and using \eqref{D4}, we deduce that
\begin{equation} \label{right part}
\int_{L_1} v \exp\left\{\frac{z}{v}- \frac{1}{2v^2} - \frac{v^2}{2} \right\} dv = \int_{E_2 + E_3} v \exp\left\{\frac{z}{v}- \frac{1}{2v^2} - \frac{v^2}{2} \right\} dv,
\end{equation}
where $L_1$ (above), $E_2$, $E_3$ are defined by
\begin{eqnarray*}
L_1:  &v& = re^{\pi i/8}, ~ r ~\text{goes~from}~ +\infty ~ \text{to} ~ 1,\\
E_2:  &v& = r, ~ r ~ \text{goes~from}~ +\infty ~\text{to}~ 1,\\
E_3:  &v& = e^{i\theta}, ~\theta~ \text{goes~from}~ 0 ~ \text{to}~ \pi/8.
\end{eqnarray*}

We now use the same reasoning in the left half-plane. From Cauchy's theorem, for any fixed $R > 1$,
\begin{equation} \label{S(R)}
\int_{S(R)} v \exp\left\{\frac{z}{v}- \frac{1}{2v^2} - \frac{v^2}{2} \right\} dv = 0,
\end{equation}
where $S(R)$ is the closed contour defined by $S(R) = S_1 + S_2 + S_3 + S_4$ with
\begin{eqnarray*}
S_1:  &v& = re^{7\pi/8}, ~ r ~\text{goes~from}~ 1 ~ \text{to} ~ R,\\
S_2:  &v& = Re^{i\theta}, ~\theta~ \text{goes~from}~ 7\pi/8 ~ \text{to}~ \pi,\\
S_3:  &v& = -r, ~ r ~ \text{goes~from}~ R ~\text{to}~ 1,\\
S_4:  &v& = e^{i\theta}, ~ \theta ~ \text{goes~from}~ \pi ~ \text{to}~ 7\pi/8.
\end{eqnarray*}

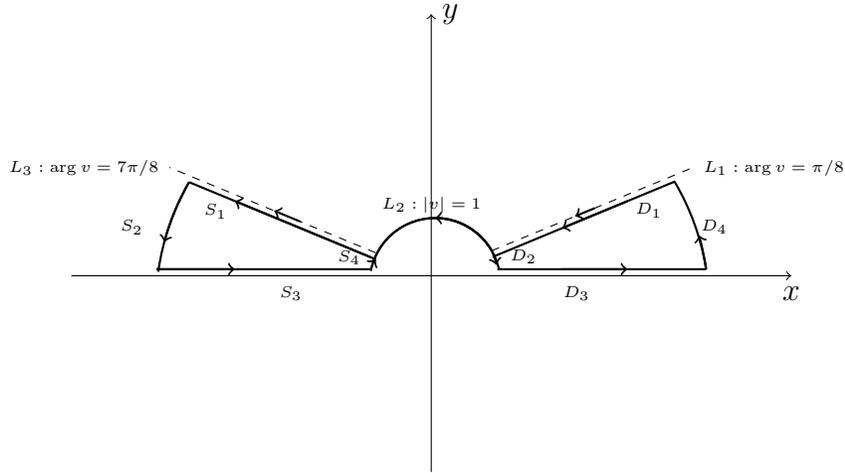
\begin{figure}[H]\label{contours D(R) and S(R)}
    \begin{center}
    \begin{tikzpicture}[scale=0.87]
    \draw[->](-5.5,0)--(5.5,0)node[left,below]{$x$};
    \draw[->](0,-3)--(0,4)node[right]{$y$};
    \draw[-,dashed][domain=0.924:4] plot(\x,{0.414*\x})node[right,font=\tiny]{$L_1: \arg v=\pi/8$};
    \draw[-,dashed][domain=-0.85:-4] plot(\x,{-0.414*\x})node[left,font=\tiny]{$L_3: \arg v=7\pi/8$};
    \draw[thick][->](2.5,2.5*0.414)--(2.2,2.2*0.414);
    \draw[thick][->](-2,2*0.414)--(-2.4,0.414*2.4);
     \draw[thick][->](0.924,0.924*0.414)arc(30:90:1);
     \draw[thick](0.924,0.924*0.414)arc(30:169:1)node[above=25pt,right,font=\tiny]{$L_2: |v|=1$};
    \draw[thick](-4.2,0.1)--(-0.93,0.1)node[left=30pt,below=2pt,font=\tiny]{$S_3$};
    \draw[thick][->](-4.2,0.1)--(-3,0.1);    \draw[thick](0.924,0.924*0.414)arc(30:165:1)node[above=3.5pt,left,font=\tiny]{$S_4$};
    \draw[thick][<-](-0.85,0.85*0.414-0.1)arc(145:146:1);
     \draw[thick](-0.85,0.85*0.414-0.1)--(-3.7,3.7*0.414-0.1)node[right=10pt,below=4pt,font=\tiny]{$S_1$};
     \draw[thick][->](-0.924,0.924*0.414-0.1)--(-3,3*0.414-0.1);
     \draw[thick](-3.7,3.7*0.414-0.1)arc(180-30:180-8:3.8)node[left=10pt,above=10pt,font=\tiny]{$S_2$};
      \draw[thick][->](-3.7,3.7*0.414-0.1)arc(180-30:180-15:3.8);
    \draw[thick](4.2,0.1)--(1.01,0.1)node[right=30pt,below=2pt,font=\tiny]{$D_3$};
    \draw[thick][->](2,0.1)--(3,0.1);
     \draw[thick](0.924,0.924*0.414)arc(30:12:1)node[above=5pt,right,font=\tiny]{$D_2$};
     \draw[thick][->](0.924,0.924*0.414)arc(25:12:1);
     \draw[thick](0.938,0.938*0.414-0.1)--(3.73,3.73*0.414-0.1)node[left=10pt,below=4pt,font=\tiny]{$D_1$};
     \draw[thick][<-](2,2*0.414-0.1)--(3,3*0.414-0.1) ;
     \draw[thick](4.2,0.1)arc(9:30.5:3.8)node[right=15pt,below=10pt,font=\tiny]{$D_4$};
     \draw[thick][->](4.2,0.1)arc(8:16:3.8);
    \end{tikzpicture}
    \end{center}
    \begin{quote}
    \caption{Contours $D(R)$ and $S(R)$}
    \end{quote}
    \end{figure} 
From similar reasoning to that used to get \eqref{D4}, we will obtain
$$ 
\int_{S_2} v \exp\left\{\frac{z}{v}- \frac{1}{2v^2} - \frac{v^2}{2} \right\} dv \to 0 \quad \hbox{as} \quad R \to \infty.
$$
Thus, by letting $R \to \infty$ in \eqref{S(R)}, we deduce that
\begin{equation} \label{left part}
\int_{L_3} v \exp\left\{\frac{z}{v}- \frac{1}{2v^2} - \frac{v^2}{2} \right\} dv = \int_{T_2 + T_3} v \exp\left\{\frac{z}{v}- \frac{1}{2v^2} - \frac{v^2}{2} \right\} dv,
\end{equation}
where $L_3$ (above), $T_2$, $T_3$ are defined by
\begin{eqnarray*}
L_3:  &v& = re^{7\pi i/8}, ~ r ~\text{goes~from}~ 1 ~ \text{to} ~ +\infty,\\
T_2:  &v& = e^{i\theta}, ~\theta~ \text{goes~from}~ 7\pi/8 ~ \text{to}~ \pi,\\
T_3:  &v& = -r, ~ r ~ \text{goes~from}~ 1 ~\text{to}~ +\infty.
\end{eqnarray*}
By combining \eqref{L curve}, \eqref{right part} and \eqref{left part}, we obtain \eqref{new H form}.\end{proof}

The following result exhibits relationships between $U(z)$ in \eqref{U def}, $H(z)$ in \eqref{new H form}, and $H(-z)$. Since $U(z)$ is an even function, we have $U(z) \equiv U(-z)$, which is confirmed again by \eqref{H & U} below.

\begin{theorem} \label{H and U}

The three contour integral functions $U(z)$, $H(z)$, $H(-z)$ are pairwise linearly independent solutions of \eqref{U equa} which satisfy 
\begin{equation} \label{H and U order}
\rho(U(z)) = \rho(H(z)) = \rho(H(-z)) = 2/3.
\end{equation}
Moreover, we have
\begin{equation} \label{H & U}
2 \pi i U(z) \equiv H(z) + H(-z).
\end{equation}

\end{theorem}

\begin{proof} Since $H(z)$ is a known solution of \eqref{U equa} from Example~\ref{GSW ex 7}, it can be verified that $H(-z)$ is also a solution of \eqref{U equa}. Since $U(z)$ in \eqref{U def} is a solution of \eqref{U equa}, all three functions $U(z)$, $H(z)$, $H(-z)$ are solutions of \eqref{U equa}. We have $\rho(U(z)) = 2/3$ from Theorem~\ref{psi order}. Since $\rho(H(z)) = 2/3$ from Example~\ref{GSW ex 7}, it follows that $\rho(H(-z)) = 2/3$. Thus, \eqref{H and U order} holds. 

Next, we have
$$H(-z) = \int_{A} w \exp\left\{-\frac{z}{w}- \frac{1}{2w^2} - \frac{w^2}{2} \right\} dw,$$
where $A$ is the contour in \eqref{new H form}. From the substitution $v = -w$, we obtain
\begin{equation} \label{H(-z) form} 
H(-z) = \int_{P} v \exp\left\{\frac{z}{v}- \frac{1}{2v^2} - \frac{v^2}{2} \right\} dv,
\end{equation}
where the contour $P$ is defined by $P=P_1+P_2+P_3$ with
\begin{eqnarray*}
P_1&:&  v = -r, ~ r ~\text{goes~from}~ +\infty ~ \text{to} ~ 1,\\
P_2&:&  v = e^{i\theta}, ~\theta~ \text{goes~from}~ \pi ~ \text{to}~ 2\pi,\\
P_3&:&  v = r, ~ r~ \text{goes~from}~ 1 ~\text{to}~ +\infty.
\end{eqnarray*}
The identity \eqref{H & U} can be deduced from \eqref{U def}, \eqref{new H form}, and \eqref{H(-z) form}.

It remains to prove that the three functions $U(z)$, $H(z)$, $H(-z)$ are pairwise linearly independent. First we show that $U(z)$ and $H(z)$ are linearly independent. To this end, we will compute the Wronskian $W(U(z), H(z))(0)$. Since $U(z)$ is an even function, it follows that $U'(0) = 0$. Hence, 
\begin{equation} \label{W}
W(U(z), H(z))(0) = U(0)H'(0).
\end{equation}

We compute $U(0)$ and $H'(0)$. From \eqref{U def},
\begin{equation*}
\begin{split}
U(0)&=\frac{1}{2\pi }\int_0^{2\pi}e^{2i\theta}\exp\left\{-\frac{1}{2}e^{-2i\theta}-\frac{1}{2}e^{2i\theta}\right\} \; d\theta\\
    &=\frac{1}{2\pi}\int_0^{2\pi}\cos 2\theta \; e^{-\cos 2\theta} \; d\theta + \frac{i}{2\pi}\int_0^{2\pi}\sin 2\theta \; e^{-\cos 2\theta} \; d\theta.
\end{split}
\end{equation*}
The second integral equals zero because $U(z)$ is real on the real axis from Theorem~\ref{psi order}. By using a computer calculation on the first integral, we get the approximation
\begin{equation} \label{U(0)}
U(0) \approxeq -0.5652.
\end{equation}
From \eqref{new H form}, we have
\begin{equation*}
\begin{split}
H'(z)=& \int_{A}\exp\left\{\frac{z}{w}- \frac{1}{2w^2} - \frac{w^2}{2} \right\} dw\\
     =&\int_\infty^1\exp\left\{\frac{z}{r}-\frac{1}{2r^2}-\frac{1}{2}r^{2}\right\}dr\\
      &+i\int_0^\pi e^{i\theta}\exp\left\{ze^{-i\theta}- \frac{1}{2}e^{-2i\theta} - \frac{1}{2}e^{2i\theta} \right\} d\theta\\
      &-\int^\infty_1\exp\left\{-\frac{z}{r}-\frac{1}{2r^2}-\frac{1}{2}r^{2}\right\}dr.
\end{split}
\end{equation*}
This gives
\begin{equation*}
H'(0) = i\int_0^\pi e^{i\theta}\exp\left\{- \cos 2\theta \right\} d\theta
    -2\int^\infty_1\exp\left\{-\frac{1}{2r^2}-\frac{1}{2}r^{2}\right\}dr,
\end{equation*}
which reduces to
$$H'(0) = -\int_0^\pi \sin \theta \; e^{-\cos 2\theta} \, d\theta
  -2\int^\infty_1\exp\left\{-\frac{1}{2r^2}-\frac{1}{2}r^{2}\right\}dr,$$
because
$$\int_0^\pi \cos \theta \; e^{-\cos 2\theta} \, d\theta = -\int_{-\pi/2}^{\pi/2} \sin\theta \; e^{\cos2\theta} \, d\theta=0.$$
Hence, $H'(0) \not= 0$, which when combined with \eqref{W} and \eqref{U(0)}, gives
$$W(U(z), H(z))(0) \not= 0.$$
It follows that $U(z)$ and $H(z)$ are linearly independent. Thus, since \eqref{H & U} holds, it can be deduced that the three functions $U(z)$, $H(z)$, $H(-z)$ are pairwise linearly independent. This completes the proof of Theorem~\ref{H and U}.\end{proof}

\bigskip
\noindent
\emph{Gary~G.~Gundersen}\\
\textsc{University of New Orleans, Department of Mathematics, New Orleans, LA 70148, USA}\\
\texttt{email:ggunders@uno.edu}
\medskip

\noindent
\emph{Janne~M.~Heittokangas}\\
\textsc{University of Eastern Finland, Department of Physics and Mathematics,
P.O.~Box 111, 80101 Joensuu, Finland}\\
\textsc{Taiyuan University of Technology,
Department of Mathematics,
Yingze West Street, No. 79, Taiyuan 030024, China}\\
\texttt{email:janne.heittokangas@uef.fi}
\medskip

\noindent
\emph{Zhi-Tao~Wen}\\
\textsc{Shantou University,
Department of Mathematics,
Daxue Road, No. 243, Shantou 515063, China}\\
\textsc{Taiyuan University of Technology,
Department of Mathematics,
Yingze West Street, No. 79, Taiyuan 030024, China}\\
\texttt{e-mail:zhtwen@stu.edu.cn}

\end{document}